\documentclass[12pt]{article}
\usepackage{graphicx}
\usepackage[ruled,vlined]{algorithm2e}
\usepackage{appendix}
\usepackage{amsmath}
\usepackage{amssymb}
\usepackage{amsthm}
\usepackage{multirow}
\usepackage{color}
\usepackage{longtable}
\usepackage{array}
\usepackage{url}
\usepackage{booktabs}
\usepackage{float}
\usepackage{mathtools}
\usepackage{tikz}
\usepackage{caption}

\newtheorem{theorem}{Theorem}[section]
\newtheorem{definition}[theorem]{Definition}
\newtheorem{lemma}[theorem]{Lemma}

\newtheorem{corollary}[theorem]{Corollary}

\theoremstyle{definition}
\newtheorem{remark}[theorem]{Remark}
\newtheorem{program}{Program}

\title{$3$-uniform hypergraphs with few Berge paths of length three between any two vertices}
\author{Tao Zhang$^{\text{a,}}$\thanks{Research supported by the National Natural Science Foundation of China under Grant No. 11801109.},~ Zixiang Xu$^{\text{b}}$~ and  Gennian Ge$^{\text{b,}}$\thanks{Corresponding author (e-mail: gnge@zju.edu.cn). Research supported by the National Natural Science Foundation of China under Grant Nos. 11431003, 61571310 and 11971325, Beijing Scholars
Program, Beijing Hundreds of Leading Talents Training Project of Science and Technology, and Beijing Municipal Natural Science Foundation.}\\
\footnotesize $^{\text{a}}$ School of Mathematics and Information Science, Guangzhou University, Guangzhou 510006, China.\\
\footnotesize $^{\text{b}}$ School of Mathematics Sciences, Capital Normal University, Beijing 100048, China.\\
}

\begin{document}

\date{}

\maketitle

\begin{abstract}
Recently, Berge theta hypergraphs have received special attention due to the similarity with Berge even cycles. Let $r$-uniform Berge theta hypergraph $\Theta_{\ell,t}^{B}$ be the $r$-uniform hypergraph consisting of $t$ internally disjoint Berge paths of length $\ell$ with the same pair of endpoints.  In this work, we determine the Tur\'{a}n number of $3$-uniform Berge theta hypergraph when $\ell=3$ and $t$ is relatively small. More precisely, we provide an explicit construction giving
\begin{align*}
\textup{ex}_{3}(n,\Theta_{3,217}^{B})=\Omega(n^{\frac{4}{3}}).
\end{align*}
This matches an earlier upper bound by He and Tait up to an absolute constant factor. The construction is algebraic, which is based on some equations over finite fields, and the parameter $t$ in our construction is much smaller than that in random algebraic construction.  Our main technique is using the resultant of polynomials,  which appears to be a powerful  technique to eliminate variables.

\medskip

\noindent {{\it Key words and phrases\/}: Tur\'{a}n number, algebraic construction, Berge theta hypergraph.}

\smallskip

\noindent {{\it AMS subject classifications\/}: 05C35, 05C65.}
\end{abstract}

\section{Introduction}
Tur\'{a}n problem is one of the most important problems in extremal graph theory. The Tur\'{a}n number $\text{ex}(n,H)$ is the maximum number of edges a graph with $n$ vertices can have that contains no copy of $H$ as a subgraph. This function was first studied by Mantel \cite{1907Mantel} and Tur\'{a}n \cite{1941Turan}, who determined the precise value of $\text{ex}(n,H)$ when $H$ is a complete graph. For general graph $H,$ Erd\H{o}s and Stone \cite{1946ErodsBAMS} gave that
\begin{align*}
\text{ex}(n,H)=(1-\frac{1}{\chi(H)-1}+o(1))\binom{n}{2},
\end{align*}
which asymptotically solves the problem when $\chi(H)\ge3$. However, it is difficult to determine the exact asymptotic results for $\textup{ex}(n,H)$ when $H$ is a bipartite graph.

Let $C_{2\ell}$ be an even cycle, the extremal results of $\textup{ex}(n,C_{2\ell})$ were first studied by Erd\H{o}s \cite{1938erdos}, and since then the problem of determining $\textup{ex}(n,C_{2\ell})$ has become a central problem in extremal graph theory.
In 1974, Bondy and Simonovits \cite{BondyEvenCycle1974} gave a general upper bound $\textup{ex}(n,C_{2\ell})\le 100\ell n^{1+\frac{1}{\ell}}.$ Recently, Bukh and Jiang \cite{ZilinCPC2017} improved the upper bound to $\textup{ex}(n,C_{2\ell})\le 80\sqrt{\ell}\log{\ell} n^{1+\frac{1}{\ell}}$, which is the best known upper bound. However, besides $C_{4},C_{6}$ and $C_{10}$, the order of magnitude for $\textup{ex}(n,C_{2\ell})$ is still unknown, see \cite{Brown1966,Erdos1966,WengerC4C6C10}. For general $\ell$, the best known lower bounds for $\textup{ex}(n,C_{2\ell})$ (except for $\textup{ex}(n,C_{14})$ \cite{2012EX14}) were due to Lazebnik, Ustimenko and Woldar \cite{1995LowerBoundC2l}.

A theta graph $\Theta_{\ell,t}$ is a graph made of $t$ internally disjoint paths of length $\ell$ connecting two endpoints. A $\Theta_{\ell,2}$ is an even cycle $C_{2\ell}$, so the problem of determining  $\textup{ex}(n,\Theta_{\ell,t})$ is a generalization of determining $\textup{ex}(n,C_{2\ell})$.
In 1983, Faudree and Simonovits \cite{Faudree1983} showed the general upper bound $\textup{ex}(n,\Theta_{\ell,t})=O_{\ell,t}(n^{1+\frac{1}{\ell}})$. Note that $\text{ex}(n,C_{2\ell})=\Theta(n^{1+\frac{1}{\ell}})$ when $\ell\in\{2,3,5\}$, then we have $\text{ex}(n,\Theta_{\ell,t})=\Theta(n^{1+\frac{1}{\ell}})$ for $\ell\in\{2,3,5\}$ and any $t\ge2$. Recently, Verstra\"{e}te and Williford \cite{2019theta43} showed that $\textup{ex}(n,\Theta_{4,3})= \Theta(n^{\frac{5}{4}})$.
 For general $\ell$, using random algebraic method, Conlon \cite{Conlon2014} showed the matched lower bound when $t$ is a sufficiently large constant. After that Bukh and Tait \cite{BukhTailtheta2018} studied the behavior of $\textup{ex}(n,\Theta_{\ell,t})$ when $\ell$ is fixed and $t$ is relatively large, and determined the dependence on $t$ when $\ell$ is odd.

By contrast with the simple graph cases, only a few results are known for hypergraph Tur\'{a}n problems. The classical definition of a hypergraph cycle is due to Berge. A Berge cycle of length $k$ is an alternating sequence of distinct vertices and distinct edges of the form $v_{1},h_{1},v_{2},h_{2},\dots,v_{k},h_{k},v_{1}$, where $v_{i},v_{i+1}\in h_{i}$ for $i\in\{1,2,\dots,k-1\}$ and $v_{k},v_{1}\in h_{k}$. A Berge path is defined similarly. The Tur\'{a}n number of Berge-$C_{k}$ is denoted by $\text{ex}(n,C_{k}^{B})$. Lazebnik and Verstra\"{e}te \cite{LV2003} studied the maximum number of edges in an $r$-uniform hypergraph containing no Berge cycle of length less than five. $\textup{ex}_{3}(n,C_{3}^{B})$ was determined by Gy\H{o}ri \cite{Gyori2006CPC}.
In \cite{Bollobas2008DM}, Bollob\'{a}s and Gy\H{o}ri showed that $\textup{ex}_{3}(n,C_{5}^{B})=O(n^{\frac{3}{2}}).$ Gy\H{o}ri and Lemons \cite{Lemons2012CPC} proved the general upper bounds $\textup{ex}_{r}(n,C_{2\ell}^{B})=O(n^{1+\frac{1}{\ell}})$ and $\textup{ex}_{r}(n,C_{2\ell+1}^{B})=O(n^{1+\frac{1}{\ell}})$ for all $\ell\ge 2$ and $r\ge 3.$ In \cite{ImproveBergeK2t}, Gerbner, Methuku and Vizer proved that $\textup{ex}_{r}(n,C_{4}^{B})=\Theta(n^{\frac{3}{2}})$ when $2\le r\le 6$.  For more extremal results of Berge cycles, we refer the readers to \cite{GerbnerSIAM2017,JiangJCTB2018,Jacques2016} and the references therein.

Since there are only a few results on $\textup{ex}_{r}(n,C_{2\ell}^{B})$, we are interested in the generalization of theta graphs to hypergraphs. Let $r$-uniform Berge theta hypergraph $\Theta_{\ell,t}^{B}$ be a set of distinct vertices $x,y,v_{1}^{1},\cdots,v_{\ell-1}^{1},\cdots,v_{1}^{t},$ $\cdots,v_{\ell-1}^{t}$ and a set of distinct edges $e_{1}^{1},\cdots,$ $e_{\ell}^{1},\cdots,e_{1}^{t},\cdots,e_{\ell}^{t}$ such that $\{x,v_{1}^{i}\}\subset e_{1}^{i},$ $\{v_{j-1}^{i},v_{j}^{i}\}\subset e_{j}^{i}$ and $\{v_{\ell-1}^{i},y\}\subset e_{\ell}^{i}$ for $1\le i\le t$ and $2\le j\le \ell-1.$ The Tur\'{a}n number of $r$-uniform $\Theta_{\ell,t}^{B}$ is denoted by $\textup{ex}_{r}(n,\Theta_{\ell,t}^{B})$. Recently, He and Tait \cite{TailBergetheta2018} gave the following upper bound \begin{align*}
\textup{ex}_{r}(n,\Theta_{\ell,t}^{B})\le c_{r,\ell,t}n^{1+\frac{1}{\ell}},
\end{align*}
where $c_{r,\ell,t}$ is a constant depending on $r,\ell,t$. They also showed that
 \begin{align*}
 \textup{ex}_{r}(n,\Theta_{\ell,t}^{B})=\Omega_{\ell,r}(n^{1+\frac{1}{\ell}}),
 \end{align*}
 where $t$ is sufficiently large.  As far as we know, there is no asymptotically optimal construction of $\textup{ex}_{r}(n,\Theta_{\ell,t}^{B})$ for $r\ge3$, relatively small $\ell$ and $t$ except $\ell=2$ \cite{ImproveBergeK2t,Timmons2017}.
 In this paper, we consider the case $\ell=3$, and prove the following result.
\begin{theorem}\label{mainthm}
$\textup{ex}_{3}(n,\Theta_{3,217}^{B})=\Omega(n^{\frac{4}{3}})$.
\end{theorem}

The best possible value of $t$ from \cite{TailBergetheta2018} is $t=\ell^{O(\ell^{2})}$. In particular, when $\ell=3$, their construction indicated that $t\approx3^{20}$, and the random algebraic method fails well short of this due to the Lang-Weil bound \cite{LangWeil1954}.
Combining with the above upper bound, we have
\begin{corollary}
$\textup{ex}_{3}(n,\Theta_{3,217}^{B})=\Theta(n^{\frac{4}{3}})$.
\end{corollary}
This paper is organized as follows. In Section~\ref{pre}, we will give some basics about the resultant of polynomials. In Section~\ref{mainpart}, we prove our main result. Section~\ref{conclusion} concludes our paper. All computations have been done by MAGMA \cite{BCP97}.
\section{Preliminaries}\label{pre}
In this section, we recall some basics about the resultant of polynomials, which will be used in the following section. Let $\mathbb{F}$ be a field, and $\mathbb{F}[x]$ be the polynomial ring with coefficients in $\mathbb{F}$.
\begin{definition}
Let $f(x),g(x)\in\mathbb{F}[x]$ with $f(x)=a_{m}x^{m}+\cdots+a_{1}x+a_{0}$ and $g(x)=b_{n}x^{n}+\cdots+b_{1}x+b_{0}$, then the resultant of $f$ and $g$ is defined by the determinant of the following $(m+n+2)\times (m+n+2)$ matrix,
\begin{align*}
\left(
            \begin{array}{cccccccc}
              a_{0} & a_{1} & \cdots & a_{m} &   &   &&   \\
               & a_{0} & \cdots  & a_{m-1} & a_{m} && & \\
                &   & \cdots  & \cdots  & \cdots & &  &\\
              &  &   &   &   & a_{0} & \cdots  & a_{m}\\
              b_{0} & b_{1} & \cdots &\cdots & b_{n} &     &\\
               & b_{0} & \cdots & \cdots  & b_{n-1} & b_{n} & & \\
                &   &  \cdots &  \cdots &\cdots  & \cdots  &   &\\
                &   &   &   &b_{0} & \cdots& \cdots  & b_{n}\\
            \end{array}
          \right),
\end{align*}
which is denoted by $R(f,g)$.
\end{definition}
The resultant of two polynomials has the following property.
\begin{lemma}\cite{Fuhrmann2012}
If $\text{gcd}(f(x),g(x))=h(x)$, where $\text{deg}(h(x))\ge1$, then $R(f,g)=0$. In particular, if $f$ and $g$ have a common root in $\mathbb{F}$, then $R(f,g)=0$.
\end{lemma}

When we consider multivariable polynomials, we can define the resultant similarly (regard the coefficients $a_{i}$ and $b_{i}$ as polynomials), and the above lemma still holds when we fix one variable.

Let $f,g\in\mathbb{F}[x_{1},\dots,x_{n}]$, we use $R(f,g,x_{i})$ to denote the resultant of $f$ and $g$ with respect to $x_{i}$, then $R(f,g,x_{i})\in\mathbb{F}[x_{1},\dots,x_{i-1},x_{i+1},\dots,x_{n}]$.
\section{Construction of $\Theta_{3,217}^{B}$-free hypergraphs}\label{mainpart}
In this section, we describe an algebraic construction of $3$-uniform hypergraph with no copy of $\Theta_{3,217}^{B}$ having $n$ vertices and $\Omega(n^{\frac{4}{3}})$ edges.

Let $p$ be a sufficiently large prime number, $\mathbb{F}_{p}$ be the finite field of order $p$. Let
\begin{align*}
&T_{1}=\{x: x\in[2,\frac{p-1}{2}]\},\\
&T_{2}=\{x: x\in[\frac{p+3}{2},p-1]\},\\
&T_{3}=\mathbb{F}_{p}\backslash\{-x^{2}:x\in\mathbb{F}_{p}\},\\
&T_{4}=\{x: x\in \mathbb{F}_{p}, x^{2}-4x+1=0\}\cup\{x: x\in \mathbb{F}_{p}, 3x-1=0\}\cup\{x: x\in \mathbb{F}_{p}, 3x-2=0\},\\
&T_{5}=\{x: x\in \mathbb{F}_{p}, x^5 - \frac{12757}{10872} x^4 + \frac{1123}{3624}x^3 + \frac{289}{1359}x^2 - \frac{49}{453}x -\frac{2}{151}=0\},
 \end{align*}
 where $\frac{1}{a}$ means the inverse of $a$ in $\mathbb{F}_{p}$. Since $p$ is a sufficiently large prime, then the above definition of $T_{i}$ is well-defined.
 Note that $T_{1}\cup T_{2}\cup\{0,1,\frac{p+1}{2}\}=\mathbb{F}_{p}$ and $|T_{3}|=\frac{p-1}{2}$. Then there exists $1\le i\le2$ such that $|T_{i}\cap T_{3}|\ge\frac{p-7}{4}$. Without loss of generality, we assume that $|T_{1}\cap T_{3}|\ge\frac{p-7}{4}$. Let $S_{1}=(T_{1}\cap T_{3})\backslash (T_{4}\cup T_{5})$, $S_{2}=\mathbb{F}_{p}\backslash\{0,1\}$. Since $|T_{4}|\le4$ and $|T_{5}|\le5$, then $|S_{1}|\ge\frac{p-43}{4}$.

Now we are ready to construct a $3$-partite $3$-uniform hypergraph as follows.
 For $1\le i\le 3$, let $V_{i}=S_{1}\times S_{2}\times S_{2}\times \{i\}$. The union $V_{1}\cup V_{2}\cup V_{3}$ will be the vertex set of our hypergraph. Given $x_{1},x_{2},x_{3}\in S_{1}$ and $a\in\mathbb{F}_{p}^{*}$, let
\begin{align*}
e(x_{1},x_{2},x_{3},a)=\{(x_{1},x_{2}x_{3}+a,x_{2}^{2}x_{3}+a,1),(x_{2},x_{3}x_{1}+a,x_{3}^{2}x_{1}+a,2),(x_{3},x_{1}x_{2}+a,x_{1}^{2}x_{2}+a,3)\}.
\end{align*}
\begin{definition}
We define $\mathcal{H}$ to be the $3$-uniform hypergraph with vertex set
\begin{align*}
V(\mathcal{H})=\{(b,c,d,i): b\in S_{1},c,d\in S_{2},1\le i\le3\}
\end{align*}
and edge set
\begin{align*}
E(\mathcal{H})=\{e(x_{1},x_{2},x_{3},a): e(x_{1},x_{2},x_{3},a)\subseteq V(\mathcal{H})\}.
\end{align*}
\end{definition}
It is easy to see that the number of vertices of $\mathcal{H}$ is $n:=3|S_{1}|(p-2)^{2}$, and there are at least $|S_{1}|^{3}(p-13)=\Omega(n^{\frac{4}{3}})$ edges in $\mathcal{H}$. In the following of this section, we will prove that $\mathcal{H}$ is $\Theta_{3,217}^{B}$-free.

We call a Berge $3$-path $v_{1},e_{1},v_{2},e_{2},v_{3},e_{3},v_{4}$ of type $(i_{1},i_{2},i_{3},i_{4})$ if $v_{j}\in V_{i_{j}}$ for $1\le j\le 4$, and a Berge $4$-cycle $v_{1},e_{1},v_{2},e_{2},v_{3},e_{3},v_{4} ,e_{4},v_{1}$ of type $(i_{1},i_{2},i_{3},i_{4})$ if $v_{j}\in V_{i_{j}}$ for $1\le j\le 4$.

By symmetry, without loss of generality, we only need to consider three types of Berge $3$-paths: $(1,2,1,2)$-type, $(1,2,3,1)$-type and $(1,2,3,2)$-type.
In the following, we divide our discussions into three subsections according to the types.

\subsection{$(1,2,1,2)$-type Berge $3$-paths}
We first give the following lemma.
\begin{lemma}\label{lemma1}
There is no $(1,2,1,2)$-type Berge $4$-cycle in $\mathcal{H}$.
\end{lemma}
\begin{proof}
Assume to the contrary, suppose $(u_{1},v_{1},w_{1},1),e(x_{1},x_{2},x_{3},a_{1}),(u_{2},v_{2},w_{2},2),e(y_{1},y_{2},y_{3},a_{2}),\\(u_{3},v_{3},w_{3},1),e(z_{1},z_{2},z_{3},a_{3}),(u_{4},v_{4},w_{4},2),e(t_{1},t_{2},t_{3},a_{4}),(u_{1},v_{1},w_{1},1)$ form a Berge $4$-cycle. Then by the definition of hypergraph $\mathcal{H}$, we have
\begin{align}
&u_{1}=x_{1}=t_{1},\label{eq1}\\
&v_{1}=x_{2}x_{3}+a_{1}=t_{2}t_{3}+a_{4},\label{eq2}\\
&w_{1}=x_{2}^{2}x_{3}+a_{1}=t_{2}^{2}t_{3}+a_{4},\label{eq3}\\
&u_{2}=x_{2}=y_{2},\label{eq4}\\
&v_{2}=x_{3}x_{1}+a_{1}=y_{3}y_{1}+a_{2},\label{eq5}\\
&w_{2}=x_{3}^{2}x_{1}+a_{1}=y_{3}^{2}y_{1}+a_{2},\label{eq6}\\
&u_{3}=y_{1}=z_{1},\label{eq7}\\
&v_{3}=y_{2}y_{3}+a_{2}=z_{2}z_{3}+a_{3},\label{eq8}\\
&w_{3}=y_{2}^{2}y_{3}+a_{2}=z_{2}^{2}z_{3}+a_{3},\label{eq9}\\
&u_{4}=z_{2}=t_{2},\label{eq10}\\
&v_{4}=z_{3}z_{1}+a_{3}=t_{3}t_{1}+a_{4},\label{eq11}\\
&w_{4}=z_{3}^{2}z_{1}+a_{3}=t_{3}^{2}t_{1}+a_{4}.\label{eq12}
\end{align}
We can compute to get the following equations.
\begin{align*}
&f_1:=x_2x_3-x_3x_1+t_3x_1-z_2t_3+z_2z_3-z_3y_1+y_3y_1-x_2y_3=0,\\
&f_2:=x_2x_3-x_2^2x_3-z_2t_3+z_2^2t_3=0,\\
&f_3:=x_3x_1-x_3^2x_1-y_3y_1+y_3^2y_1=0,\\
&f_4:=z_3y_1-z_3^2y_1-t_3x_1+t_3^2x_1=0,\\
&f_5:=x_2y_3-x_2^2y_3-z_2z_3+z_2^2z_3=0,
\end{align*}
where $f_{1}$ is from Eqs. (\ref{eq1}), (\ref{eq2}), (\ref{eq4}), (\ref{eq5}), (\ref{eq7}), (\ref{eq8}), (\ref{eq10}) and (\ref{eq11}), $f_{2}$ is from Eqs. (\ref{eq2}), (\ref{eq3}) and (\ref{eq10}), $f_{3}$ is from Eqs. (\ref{eq5}) and (\ref{eq6}), $f_{4}$ is from Eqs. (\ref{eq1}), (\ref{eq7}) (\ref{eq11}) and (\ref{eq12}), and $f_{5}$ is from Eqs. (\ref{eq4}), (\ref{eq8}) and (\ref{eq9}).

Regarding $f_{i}$ as polynomials with variables $x_1,x_2,x_3,y_1,y_3,z_2,z_3,t_3$, we can compute to get the following polynomials.
\begin{align*}
 &g_{1}=R(f_{3},f_{4},x_{1}),\\
 &g_{2}=R(f_{2},f_{5},x_{2}),\\
 &h=R(g_{1},g_{2},x_{3}).
 \end{align*}
By a MAGMA program, the polynomial $h$ can be factorized as
\begin{align*}
h=t_{3}^{2}z_{3}^{2}z_{2}^{4}y_{3}^{2}y_{1}^{2}(z_{3}-t_{3})^{2}(z_{2}-1)^{4}(y_{3}-z_{3})^{2}.
\end{align*}

{\bf{Claim: $z_{2}\ne1$, $z_{3}\ne t_{3}$ and $y_{3}\ne z_{3}$.}}

If $z_{2}=1$, then $u_{4}=1$, which contradicts to the fact that $(u_{4},v_{4},w_{4},2)\in V_{2}$ by the definition of $S_{1}$.

If $z_{3}=t_{3}$, then from Eqs. (\ref{eq11}) and (\ref{eq12}), we have $z_{1}(z_{3}-z_{3}^{2})=t_{1}(z_{3}-z_{3}^{2})$. Note that $z_{3}-z_{3}^{2}\ne0$, then $z_{1}=t_{1}$, and so $a_{3}=a_{4}$. This leads to $e(z_{1},z_{2},z_{3},a_{3})=e(t_{1},t_{2},t_{3},a_{4})$, which is a contradiction.

If $y_{3}=z_{3}$, then from Eqs. (\ref{eq8}) and (\ref{eq9}), we have $y_{2}-y_{2}^{2}=z_{2}-z_{2}^{2}\ne0$. Since $x_{2}=y_{2}$ and $z_{2}=t_{2}$, then $x_{2}-x_{2}^{2}=t_{2}-t_{2}^{2}$. From Eqs. (\ref{eq2}) and (\ref{eq3}), we have $x_{3}=t_{3}$. Substituting the equations $x_{3}=t_{3}$ and $y_{3}=z_{3}$ to $f_{1}$, we get that
\begin{align*}
(x_{2}-z_{2})(x_{3}-y_{3})=0.
\end{align*}
If $x_{3}=y_{3}$, then we have $z_{3}=t_{3}$, which is impossible. If $x_{2}=z_{2}$, then $x_{2}=t_{2}$, and so $a_{1}=a_{4}$, this leads to $e(x_{1},x_{2},x_{3},a_{1})=e(t_{1},t_{2},t_{3},a_{4})$, which is a contradiction. This completes the proof of our claim.

It is easy to see that $t_{3},z_{3},z_{2},y_{3},y_{1}\ne0$, hence $h\ne0$. On the other hand,
$h$ is obtained from $f_{i}$, so $h$ must be $0$, which is a contradiction. Hence there is no $(1,2,1,2)$-type Berge $4$-cycle in $\mathcal{H}$.
\end{proof}
\begin{remark}
In the Appendix, we give the MAGMA program for the computations in the proof of Lemma~\ref{lemma1}. The programs for the remaining computations in this paper are similar. The interested readers can ask for a copy of  the programs by contacting the authors.
\end{remark}

Now we consider $(1,2,1,2)$-type Berge $3$-path.
For any given $(b_{1},c_{1},d_{1},1)\in V_{1},(b_{2},c_{2},d_{2},2)\in V_{2}$, suppose there exist $(u_{1},v_{1},w_{1},2)\in V_{2},(u_{2},v_{2},w_{2},1)\in V_{1}$ and $e(x_{1},x_{2},x_{3},a_{1}),e(y_{1},y_{2},y_{3},a_{2}),\\e(z_{1},z_{2},z_{3},a_{3})$ such that $(b_{1},c_{1},d_{1},1),e(x_{1},x_{2},x_{3},a_{1}),(u_{1},v_{1},w_{1},2),e(y_{1},y_{2},y_{3},a_{2}),(u_{2},v_{2},w_{2},1),\\e(z_{1},z_{2},z_{3},a_{3}),(b_{2},c_{2},d_{2},2)$ form a Berge $3$-path.

Then by the definition of hypergraph $\mathcal{H}$, we have
\begin{align*}
&b_{1}=x_{1},&&b_{2}=z_{2},\\
&c_{1}=x_{2}x_{3}+a_{1},&&c_{2}=z_{3}z_{1}+a_{3},\\
&d_{1}=x_{2}^{2}x_{3}+a_{1},&&d_{2}=z_{3}^{2}z_{1}+a_{3},\\
&u_{1}=x_{2}=y_{2},&&u_{2}=y_{1}=z_{1},\\
&v_{1}=x_{3}x_{1}+a_{1}=y_{3}y_{1}+a_{2},&&v_{2}=y_{2}y_{3}+a_{2}=z_{2}z_{3}+a_{3},\\
&w_{1}=x_{3}^{2}x_{1}+a_{1}=y_{3}^{2}y_{1}+a_{2},&&w_{2}=y_{2}^{2}y_{3}+a_{2}=z_{2}^{2}z_{3}+a_{3}.
\end{align*}
Rewriting the above equations and substituting $b_{1}=x_{1}$, $b_{2}=z_{2}$, $x_{2}=y_{2}$ and $y_{1}=z_{1}$ into other equations, we have the following equations.
\begin{align*}
&f_1:=x_2x_3+a_1-c_1=0,\\
&f_2:=x_2^2x_3+a_1-d_1=0,\\
&f_3:=z_3y_1+a_3-c_2=0,\\
&f_4:=z_3^2y_1+a_3-d_2=0,\\
&f_5:=x_3b_1+a_1-y_3y_1-a_2=0,\\
&f_6:=x_3^2b_1+a_1-y_3^2y_1-a_2=0,\\
&f_7:=x_2y_3+a_2-b_2z_3-a_3=0,\\
&f_8:=x_2^2y_3+a_2-b_2^2z_3-a_3=0.
\end{align*}

\begin{lemma}\label{lemma2}
\begin{itemize}
  \item [(1)] $b_{i},c_{i},d_{i},u_{i},v_{i},w_{i},x_{j},y_{j},z_{j}\not\in\{0,1\}$ for $1\le i\le2$ and $1\le j\le 3$,
  \item [(2)] $c_{1}\ne d_{1}$, $x_{2}\ne b_{2}$, $c_{2}\ne d_{2}$,
  \item [(3)] $x_{2}+b_{2}\ne1$, $x_{2}^{2}-x_{2}+c_{1}-d_{1}\ne0$, $b_{1}+b_{2}^{2}\ne0$.
\end{itemize}
\end{lemma}
\begin{proof}
All the statements are immediately from the definition of $\mathcal{H}$. We will only prove $x_{2}^{2}-x_{2}+c_{1}-d_{1}\ne0$. If $x_{2}^{2}-x_{2}+c_{1}-d_{1}=0$, then by the equations on $c_{1}$ and $d_{1}$, we have $c_{1}-d_{1}=(x_{2}-x_{2}^{2})x_{3}$, then $x_{3}=1$, which is a contradiction.
\end{proof}
\begin{remark}\label{rmk1}
When we consider the resultant of two polynomials, we can divide these nonzero factors first.
\end{remark}

Now we regard $f_{i}$ as polynomials with variables $x_2,x_3,y_1,y_3,z_3,a_1,a_2,a_3$. We can compute to get the following polynomials.
 \begin{align*}
 &g_{1}=R(f_{1},f_{2},a_{1}),&&g_{2}=R(f_{5},f_{6},a_{1}),\\
 &g_{3}=R(f_{1},f_{5},a_{1}),&&g_{4}=R(f_{3},f_{4},a_{3}),\\
 &g_{5}=R(f_{7},f_{8},a_{3}),&&g_{6}=R(f_{3},f_{7},a_{3}),\\
 &g_{7}=R(g_{3},g_{6},a_{2}),&&h_{1}=R(g_{1},g_{2},x_{3}),\\
 &h_{2}=R(g_{1},g_{7},x_{3}),& &h_{3}=R(g_{4},g_{5},z_{3}),\\
 &h_{4}=R(g_{5},h_{2},z_{3}),&&h_{5}=R(h_{1},h_{3},y_{1}),\\
 &h_{6}=R(h_{1},h_{4},y_{1}).
 \end{align*}
 By a MAGMA program, the polynomials $h_{5}$ and $h_{6}$ can be factorized as
\begin{align*}
&h_{5}=y_{3}x_{2}(x_{2}-1)h_{5}',\\
&h_{6}=y_{3}x_{2}(x_{2}-1)h_{6}'.
\end{align*}
Finally, we can get that
\begin{align*}
R(h_{5}',h_{6}',y_{3})=b_{1}x_{2}^{2}(c_{1}-d_{1})(x_{2}-1)^{2}(x_{2}-b_{2})(x_{2}+b_{2}-1)(x_{2}^{2}-x_{2}+c_{1}-d_{1})m,
\end{align*}
where $m$ is a polynomial of $x_{2}$ with degree $4$.

By Lemma~\ref{lemma2}, we have $b_{1}x_{2}^{2}(c_{1}-d_{1})(x_{2}-1)^{2}(x_{2}-b_{2})(x_{2}+b_{2}-1)(x_{2}^{2}-x_{2}+c_{1}-d_{1})\ne0$. Since $m$ is obtained from $f_{i}$, then $m=0$.
We write $m$ as $m=m_{4}x_{2}^{4}+m_{3}x_{2}^{3}+m_{2}x_{2}^{2}+m_{1}x_{2}+m_{0}$. We claim that $m_{4},m_{3},m_{2},m_{1},m_{0}$ cannot be all $0$. Assume to the contrary, we regard $m_{i}$ as polynomials with variables $b_{1},c_{1},d_{1},b_{2},c_{2},d_{2}$.
By a MAGMA program,
\begin{align*}m_{0}=&(-b_1c_1^2 + 2b_1c_1d_1 - b_1c_1b_2^2 + b_1c_1b_2 - b_1d_1^2 + b_1d_1b_2^2 - b_1d_1b_2 - b_2^4c_2 + b_2^4d_2 + 2b_2^3c_2 -2b_2^3d_2 -\\
 &b_2^2c_2 + b_2^2d_2)\cdot m_{0}'.
\end{align*}

{\bf{Claim} $C:=-b_1c_1^2 + 2b_1c_1d_1 - b_1c_1b_2^2 + b_1c_1b_2 - b_1d_1^2 + b_1d_1b_2^2 - b_1d_1b_2 - b_2^4c_2 + b_2^4d_2 + 2b_2^3c_2 -2b_2^3d_2 - b_2^2c_2 + b_2^2d_2\ne0.$}

\begin{proof}[Proof of the claim]
If $C=0$, then we have
\begin{align*}
R(C,m_{1},d_{2})=&b_{2}^{3}(b_{2}-1)^{3}(c_{1}-d_{1})b_{1}(b_{1}c_{1}-b_{1}d_{1}-c_{1}b_{2}^{2}+d_{1}b_{2}+b_{2}^{2}c_{2}-b_{2}c_{2}),\\
R(C,m_{1},c_{2})=&b_{2}^{4}(b_{2}-1)^{4}(c_{1}-d_{1})b_{1}(b_{1}c_{1}^{2}-2b_{1}c_{1}d_{1}+b_{1}d_{1}^{2}+c_{1}b_{2}^{4}-c_{1}b_{2}^{3}-d_{1}b_{2}^{3}+d_{1}b_{2}^{2}-b_{2}^{4}d_{2}\\
&+2b_{2}^{3}d_{2}-b_{2}^{2}d_{2}).
\end{align*}
Hence
\begin{align*}
&b_{1}c_{1}-b_{1}d_{1}-c_{1}b_{2}^{2}+d_{1}b_{2}+b_{2}^{2}c_{2}-b_{2}c_{2}=0,\\
&b_{1}c_{1}^{2}-2b_{1}c_{1}d_{1}+b_{1}d_{1}^{2}+c_{1}b_{2}^{4}-c_{1}b_{2}^{3}-d_{1}b_{2}^{3}+d_{1}b_{2}^{2}-b_{2}^{4}d_{2}+2b_{2}^{3}d_{2}-b_{2}^{2}d_{2}=0.
\end{align*}
Let $t_{1}=b_{1},t_{2}=b_{2},t_{3}=\frac{c_{1}-d_{1}}{b_{2}-b_{2}^{2}},a_{4}=\frac{d_{1}-b_{2}c_{1}}{1-b_{2}}$. Then it is easy to get that $(b_{1},c_{1},d_{1},1),(b_{2},c_{2},d_{2},2)\in e(t_{1},t_{2},t_{3},a_{4})$. Hence there exists a $(1,2,1,2)$-type Berge $4$-cycle in $\mathcal{H}$, which contradicts to Lemma~\ref{lemma1}. This completes the proof.
\end{proof}

Now we can compute to get that
\begin{align*}
R(m_{1},m_{0}',c_{1})=b_{2}^{6}b_{1}^{3}(c_{2}-d_{2})^{2}(b_{2}-1)^{6}(b_{1}+b_{2}^{2})^{2}.
\end{align*}
By Lemma~\ref{lemma2}, we have $R(m_{1},m_{0}',c_{1})\ne0$, which is a contradiction. Since $m$ is a polynomial of $x_{2}$ with degree $4$, then there are at most $4$ solutions for $x_{2}$.

Now for any fixed $x_{2}$, we consider the polynomial $h_{5}'$, which is a polynomial of $y_{3}$ with degree $1$. We write $h_{5}'$ as $h_{5}'=r_{1}y_{3}+r_{0}$, then $r_{0}$ and $r_{1}$ can be factorized as
\begin{align*}
&r_{0}=b_{2}(b_{2}-1)r_{0}',\\
&r_{1}=x_{2}(x_{2}-1)r_{1}'.
\end{align*}
We can compute to get that
\begin{align*}
R(r_{0}',r_{1}',x_{2})=C\cdot (c_{2}-d_{2})^{2}b_{2}^{2}(b_{2}-1)^{2}\ne0,
\end{align*}
where $C$ is defined in the above claim. Hence $r_{1}$ and $r_{0}$ cannot be $0$ at the same time, therefore, there is at most $1$ solution for $y_{3}$ when $x_{2}$ is given. If $x_{2}$ and $y_{3}$ are given, then all the remaining variables are uniquely determined.

Hence, for any given $(b_{1},c_{1},d_{1},1)\in V_{1},(b_{2},c_{2},d_{2},2)\in V_{2}$, there are at most $4$ Berge $3$-paths of $(1,2,1,2)$-type with $(b_{1},c_{1},d_{1},1),(b_{2},c_{2},d_{2},2)$ being its end core vertices.

\subsection{$(1,2,3,1)$-type Berge $3$-paths}

For any given $(b_{1},c_{1},d_{1},1),(b_{2},c_{2},d_{2},1)\in V_{1}$, suppose there exist $(u_{1},v_{1},w_{1},2)\in V_{2},(u_{2},v_{2},w_{2},3)\in V_{3}$ and $e(x_{1},x_{2},x_{3},a_{1}),e(y_{1},y_{2},y_{3},a_{2}),e(z_{1},z_{2},z_{3},a_{3})$ such that $(b_{1},c_{1},d_{1},1),e(x_{1},x_{2},x_{3},a_{1}),\\(u_{1},v_{1},w_{1},2),e(y_{1},y_{2},y_{3},a_{2}),(u_{2},v_{2},w_{2},3),e(z_{1},z_{2},z_{3},a_{3}),(b_{2},c_{2},d_{2},1)$ form a Berge $3$-path.

Then by the definition of hypergraph $\mathcal{H}$, we have
\begin{align*}
&b_{1}=x_{1},&&b_{2}=z_{1},\\
&c_{1}=x_{2}x_{3}+a_{1},&&c_{2}=z_{2}y_{3}+a_{3},\\
&d_{1}=x_{2}^{2}x_{3}+a_{1},&&d_{2}=z_{2}^{2}y_{3}+a_{3},\\
&u_{1}=x_{2}=y_{2},&&u_{2}=y_{3}=z_{3},\\
&v_{1}=x_{3}x_{1}+a_{1}=y_{3}y_{1}+a_{2},&&v_{2}=y_{1}y_{2}+a_{2}=z_{1}z_{2}+a_{3},\\
&w_{1}=x_{3}^{2}x_{1}+a_{1}=y_{3}^{2}y_{1}+a_{2},&&w_{2}=y_{1}^{2}y_{2}+a_{2}=z_{1}^{2}z_{2}+a_{3}.
\end{align*}
Rewriting the above equations and substituting $b_{1}=x_{1}$, $b_{2}=z_{1}$, $x_{2}=y_{2}$ and $y_{3}=z_{3}$ into other equations, we can get the following equations.
\begin{align*}
&f_1:=x_2x_3+a_1-c_1=0,\\
&f_2:=x_2^2x_3+a_1-d_1=0,\\
&f_3:=z_3y_3+a_3-c_2=0,\\
&f_4:=z_3^2y_3+a_3-d_2=0,\\
&f_5:=x_3b_1+a_1-y_3y_1-a_2=0,\\
&f_6:=x_3^2b_1+a_1-y_3^2y_1-a_2=0,\\
&f_7:=y_1x_2+a_2-b_2z_2-a_3=0,\\
&f_8:=y_1^2x_2+a_2-b_2^2z_2-a_3=0.
\end{align*}

\begin{lemma}
\begin{itemize}
  \item [(1)]  $b_{i},c_{i},d_{i},u_{i},v_{i},w_{i},x_{j},y_{j},z_{j}\not\in\{0,1\}$ for $1\le i\le2$ and $1\le j\le 3$,
  \item [(2)] $c_{1}\ne d_{1}$, $c_{2}\ne d_{2}$, $b_{1}+b_{2}\ne1$,
  \item [(3)] $b_{1}^{2}-4b_{1}+1\ne0$, $2b_{2}\ne1$, $3b_{2}\ne1$,
  \item [(4)] $b_2^5 - \frac{12757}{10872}b_2^4 + \frac{1123}{3624}b_2^3 + \frac{289}{1359}b_2^2 - \frac{49}{453}b_2 -\frac{2}{151}\ne0$.
\end{itemize}
\end{lemma}
\begin{proof}
All the statements are immediately from the definition of $\mathcal{H}$.
\end{proof}
Similar as Remark~\ref{rmk1}, when we consider the resultant of two polynomials, we can divide these nonzero factors first.

Now we regard $f_{i}$ as polynomials with variables $x_2,x_3,y_1,y_3,z_2,a_1,a_2,a_3$. We can get the following polynomials.
\begin{align*}
&g_{1}=R(f_{1},f_{2},a_{1}),&&g_{2}=R(f_{5},f_{6},a_{1}),\\
&g_{3}=R(f_{1},f_{5},a_{1}),&&g_{4}=R(f_{3},f_{4},a_{3}),\\
&g_{5}=R(f_{7},f_{8},a_{3}),&&g_{6}=R(f_{3},f_{7},a_{3}),\\
&g_{7}=R(g_{3},g_{6},a_{2}),&&h_{1}=R(g_{1},g_{2},x_{3}),\\
&h_{2}=R(g_{1},g_{7},x_{3}),&&h_{3}=R(g_{4},g_{5},z_{2}),\\
&h_{4}=R(g_{5},h_{2},z_{2}),&&h_{5}=R(h_{3},h_{4},y_{3}),\\
&h_{6}=R(h_{1},h_{4},y_{3}).
\end{align*}
 By a MAGMA program, the polynomials $h_{5}$ and $h_{6}$ can be factorized as
\begin{align*}
&h_{5}=b_2y_1x_2\cdot h_5',\\
&h_{6}=y_1x_2^2(x_2-1)^2\cdot h_6'.
\end{align*}
Finally, we can compute to get that
\begin{align}\label{eq13}
r:=R(h_{5}',h_{6}',y_{1})=b_2^5(b_2-1)^8x_2^9(x_2-1)\cdot s^{2}t,
\end{align}
where $s=x_2^3 + x_2^2c_1 + x_2^2b_2^2 - 2x_2^2b_2 - x_2^2c_2 - x_2^2 - x_2d_1 - x_2b_2^3 +2x_2b_2 + x_2c_2 - b_1c_1 + b_1d_1 + b_2^3 - b_2^2$ and $t$ is a polynomial of $x_2$ with degree $24$. We write $t$ as $t=\sum_{i=0}^{24}t_{i}x_{2}^{i}$. We claim that $t_{i}$ $(0\le i\le24)$ cannot be all $0$, otherwise, we can take 6 coefficients such as
\begin{align*}
&t_0=b_2^2(b_2-1)^2(c_1-d_1)^{10}b_{1}^{5}\cdot t_{0}',\\
&t_{1}=b_{2}(b_{2}-1)(c_{1}-d_{1})^{9}b_{1}^{5}\cdot t_{1}',\\
&t_{2}=(c_{1}-d_{1})^{8}b_{1}^{4}\cdot t_{2}',\\
&t_{3}=(c_{1}-d_{1})^{7}b_{1}^{4}\cdot t_{3}',\\
&t_{4}=(c_{1}-d_{1})^{6}b_{1}^{3}\cdot t_{4}',\\
&t_{24}=(c_2-d_2)b_2^4(b_2-1)^2\cdot t_{24}'.
\end{align*}
Now we regard $t_{i}$ as the polynomials with variables $b_{1},b_{2},c_{1},c_{2},d_{1},d_{2}$. By a MAGMA program, we can get the following polynomials.
\begin{align*}
&s_{1}=R(t_{0}',t_{1}',d_{1})=(c_{2}-d_{2})b_{2}^{2}(b_{2}-1)^{2}\cdot s_{1}',\\
&s_{2}=R(t_{0}',t_{24}',d_{1}),\\
&s_{3}=R(t_{0}',t_{2}',d_{1})=(c_{2}-d_{2})b_{2}^{3}(b_{2}-1)^{3}b_{1}\cdot s_{3}',\\
&s_{4}=R(t_{0}',t_{3}',d_{1})=(c_{2}-d_{2})b_{2}^{2}(b_{2}-1)^{3}\cdot s_{4}',\\
&s_{5}=R(t_{0}',t_{4}',d_{1})=(c_{2}-d_{2})b_{2}^{3}(b_{2}-1)^{3}b_{1}\cdot s_{5}',\\
&s_{6}=R(s_{1}',s_{2},d_{2})=(c_{1}-c_{2})\cdot s_{6}',\\
&s_{7}=R(s_{1}',s_{3}',d_{2})=(c_{1}-c_{2})b_{1}^{2}(b_{1}-1)^{2}(b_{1}-b_{2})(b_{1}+b_{2}-1)\cdot s_{7}',\\
&s_{8}=R(s_{1}',s_{4}',d_{2})=(c_{1}-c_{2})b_{1}^{4}(b_{1}-1)^{3}(b_{1}-b_{2})(b_{1}+b_{2}-1)\cdot s_{8}',\\
&s_{9}=R(s_{1}',s_{5}',d_{2})=(c_{1}-c_{2})b_{1}^{4}(b_{1}-1)^{4}(b_{1}-b_{2})(b_{1}+b_{2}-1)\cdot s_{9}'.
\end{align*}

{\bf{Claim: $b_{1}\ne b_{2}$ and $c_{1}\ne c_{2}$.}}

If $b_{1}=b_{2}$. Then we substitute this equation into $t_{0}'$ to get a polynomial $t_{0}''$, and substitute it into $t_{1}'$ to get a polynomial $t_{1}''$. Then $t_{0}''$ and $t_{1}''$ can be factorized as
\begin{align*}
&t_{0}''=b_{2}(b_{2}-1)\cdot t_{0}''',\\
&t_{1}''=b_{2}(b_{2}-1)\cdot t_{1}'''.
\end{align*}
We can compute to get that $R(t_{0}''',t_{1}''',c_{1})=(c_{2}-d_{2})(b_{2}-\frac{1}{3})(d_{1}-d_{2})$, hence $d_{1}=d_{2}$. Similarly, we can get that $c_{1}=c_{2}$. Therefore $(b_{1},c_{1},d_{1},1)=(b_{2},c_{2},d_{2},1)$, which is a contradiction.

Similarly, we can prove that $c_{1}\ne c_{2}$. This completes the proof of our claim.

Now we can compute to get that
\begin{align*}
&s_{10}=R(s_{6}',s_{7}',c_{1}),\\
&s_{11}=R(s_{7}',s_{8}',c_{1})=b_1^2(b_1-b_2)^2(b_1^2-4b_1+1)^2\cdot s_{11}',\\
&s_{12}=R(s_{7}',s_{9}',c_{1})=b_1^3(b_1-b_2)^3(b_1^2-4b_1+1)^3\cdot s_{12}',\\
&s_{13}=R(s_{10},s_{11}',b_{1}),\\
&s_{14}=R(s_{10},s_{12}',b_{1}).
\end{align*}
By a MAGMA program, the polynomials $s_{i}\ (i=13,14)$ can be factorized as
\begin{align*}
s_{13}=&b_2^{19}(b_2-1)^8(b_2-\frac{2}{3})^8(b_2-\frac{1}{3})^4(b_2^5 - \frac{12757}{10872}b_2^4 + \frac{1123}{3624}b_2^3 + \frac{289}{1359}b_2^2 - \frac{49}{453}b_2 -\frac{2}{151})^2\cdot s_{13}',\\
s_{14}=&b_2^{27}(b_2-1)^{12}(b_2-\frac{2}{3})^{12}(b_2-\frac{1}{3})^6(b_2^5 - \frac{12757}{10872}b_2^4 + \frac{1123}{3624}b_2^3 + \frac{289}{1359}b_2^2 - \frac{49}{453}b_2 -\frac{2}{151})^3\cdot s_{14}'.
\end{align*}
Finally, we have $R(s_{13}',s_{14}',b_{2})\ne0$, which is a contradiction. Hence, from Eq. (\ref{eq13}), we have that there are at most 24+3=27 solutions for $x_{2}$.

Now for any fixed $x_{2}$, $h_{6}'$ is a polynomial of $y_{1}$ with degree $4$. We write $h_{6}'$ as $h_{6}'=\sum_{i=0}^{4}l_{i}y_{1}^{i}$, then $l_{4}=b_{2}(b_{2}-1)x_{2}^{4}(x_{2}-1)^{2}\ne0$. Hence there are at most $4$ solutions for $y_{1}$. If $x_{2}$ and $y_{1}$ are given, then all the remaining variables are uniquely determined.

Hence, for any given $(b_{1},c_{1},d_{1},1),(b_{2},c_{2},d_{2},1)\in V_{1}$, there are at most $108$ Berge $3$-paths of $(1,2,3,1)$-type with $(b_{1},c_{1},d_{1},1),(b_{2},c_{2},d_{2},1)$ being its end core vertices.

\subsection{$(1,2,3,2)$-type Berge $3$-paths}

For any given $(b_{1},c_{1},d_{1},1)\in V_{1},(b_{2},c_{2},d_{2},2)\in V_{2}$, suppose there exist $(u_{1},v_{1},w_{1},2)\in V_{2},\\(u_{2},v_{2},w_{2},3)\in V_{3}$ and $e(x_{1},x_{2},x_{3},a_{1}),e(y_{1},y_{2},y_{3},a_{2}),e(z_{1},z_{2},z_{3},a_{3})$ such that $(b_{1},c_{1},d_{1},1),\\e(x_{1},x_{2},x_{3},a_{1}),(u_{1},v_{1},w_{1},2),e(y_{1},y_{2},y_{3},a_{2}),(u_{2},v_{2},w_{2},3),e(z_{1},z_{2},z_{3},a_{3}),(b_{2},c_{2},d_{2},2)$ form a Berge $3$-path.

Then by the definition of hypergraph $\mathcal{H}$, we have
\begin{align*}
&b_{1}=x_{1},&&b_{2}=z_{2},\\
&c_{1}=x_{2}x_{3}+a_{1},&&c_{2}=z_{3}z_{1}+a_{3},\\
&d_{1}=x_{2}^{2}x_{3}+a_{1},&&d_{2}=z_{3}^{2}z_{1}+a_{3},\\
&u_{1}=x_{2}=y_{2},&&u_{2}=y_{3}=z_{3},\\
&v_{1}=x_{3}x_{1}+a_{1}=y_{3}y_{1}+a_{2},&&v_{2}=y_{1}y_{2}+a_{2}=z_{1}z_{2}+a_{3},\\
&w_{1}=x_{3}^{2}x_{1}+a_{1}=y_{3}^{2}y_{1}+a_{2},&&w_{2}=y_{1}^{2}y_{2}+a_{2}=z_{1}^{2}z_{2}+a_{3}.
\end{align*}
Rewriting the above equations and substituting $b_{1}=x_{1}$, $b_{2}=z_{2}$, $x_{2}=y_{2}$ and $y_{3}=z_{3}$ into other equations, we can get the following equations.
\begin{align*}
&f_1:=x_2x_3+a_1-c_1=0,\\
&f_2:=x_2^2x_3+a_1-d_1=0,\\
&f_3:=y_3z_1+a_3-c_2=0,\\
&f_4:=y_3^2z_1+a_3-d_2=0,\\
&f_5:=x_3b_1+a_1-y_3y_1-a_2=0,\\
&f_6:=x_3^2b_1+a_1-y_3^2y_1-a_2=0,\\
&f_7:=y_1x_2+a_2-z_1b_2-a_3=0,\\
&f_8:=y_1^2x_2+a_2-z_1^2b_2-a_3=0.
\end{align*}

\begin{lemma}
\begin{itemize}
  \item [(1)]  $b_{i},c_{i},d_{i},u_{i},v_{i},w_{i},x_{j},y_{j},z_{j}\not\in\{0,1\}$ for $1\le i\le2$ and $1\le j\le 3$,
  \item [(2)] $c_{1}\ne d_{1}$, $c_{2}\ne d_{2}$, $x_{2}\ne b_{2}$,
  \item [(3)] $x_{2}^{2}-x_{2}+c_{1}-d_{1}\ne0$.
\end{itemize}
\end{lemma}
\begin{proof}
We will only prove $x_{2}\ne b_{2}$, others are immediately from the definition of $\mathcal{H}$. If $x_{2}=b_{2}$, then $y_{2}=z_{2}$. By the equations on $v_{2}$ and $w_{2}$, we have $y_{1}=z_{1}$, and then $a_{2}=a_{3}$. This leads to $e(y_{1},y_{2},y_{3},a_{2})=e(z_{1},z_{2},z_{3},a_{3})$, which is a contradiction.
\end{proof}
Similar as Remark~\ref{rmk1}, when we consider the resultant of two polynomials, we can divide these nonzero factors first.

Regarding $f_{i}$ as polynomials with variables $x_2,x_3,y_1,y_3,z_1,a_1,a_2,a_3$, we can compute to get that
 \begin{align*}
&g_{1}=R(f_{1},f_{2},a_{1}),&&g_{2}=R(f_{5},f_{6},a_{1}),\\
&g_{3}=R(f_{1},f_{5},a_{1}),&&g_{4}=R(f_{3},f_{4},a_{3}),\\
&g_{5}=R(f_{7},f_{8},a_{3}),&&g_{6}=R(f_{3},f_{7},a_{3}),\\
&g_{7}=R(g_{3},g_{6},a_{2}),&&h_{1}=R(g_{1},g_{2},x_{3}),\\
&h_{2}=R(g_{1},g_{7},x_{3}),&&h_{3}=R(g_{4},g_{5},z_{1}),\\
&h_{4}=R(g_{4},h_{2},z_{1}),&&h_{5}=R(h_{1},h_{3},y_{1}),\\
&h_{6}=R(h_{1},h_{4},y_{1}).
\end{align*}
By a MAGMA program, the polynomials $h_{5}$ and $h_{6}$ can be factorized as
\begin{align*}
&h_{5}=y_{3}^{2}(y_{3}-1)^{2}x_{2}\cdot h_5',\\
&h_{6}=y_{3}(y_{3}-1)x_{2}(x_{2}-1)\cdot h_6'.
\end{align*}
Finally, we have
\begin{align*}
r=R(h_{5}',h_{6}',y_{3})=x_{2}^{2}(x_{2}-1)^{2}\cdot t,
\end{align*}
where $t$ is a polynomial of $x_2$ with degree $18$. We write $t$ as $t=\sum_{i=0}^{18}t_{i}x_{2}^{i}$. We can compute to get that
\begin{align*}
t_{0}=(c_{1}-d_{1})^{10}b_{1}^{5}(b_{1}-1)\ne0.
\end{align*}
Hence, there are at most $18$ solutions for $x_{2}$.

Now for any fixed $x_{2}$, $h_{5}'$ is a polynomial of $y_{3}$ with degree $2$. We write $h_{5}'$ as $h_{5}'=\sum_{i=0}^{2}l_{i}y_{3}^{i}$, then $l_{2}=x_{2}^{2}(x_{2}-1)^{2}l_{2}'$. We can regard $l_{i}$ as polynomials with variables $x_{2},b_{1},b_{2},c_{1},c_{2},d_{1},d_{2}$. Then by a MAGMA program, we have
\begin{align*}
R(l_{0},l_{2}',b_{1})=(c_2-d_2)^2b_2(c_1-d_1)^2x_2^2(x_2-1)^4(x_2-b_2)(x_2^2-x_2+c_1-d_1)^2\ne0.
\end{align*}
Hence there are at most $2$ solutions for $y_{3}$. If $x_{2}$ and $y_{3}$ are given, then all the remaining variables are uniquely determined.

Hence, for any given $(b_{1},c_{1},d_{1},1)\in V_{1},(b_{2},c_{2},d_{2},2)\in V_{2}$, there are at most $36$ Berge $3$-paths of $(1,2,3,2)$-type with $(b_{1},c_{1},d_{1},1),(b_{2},c_{2},d_{2},2)$ being its end core vertices.

\subsection{Proof of Theorem~\ref{mainthm}}

If the given two vertices are in the same part, without loss of generality, suppose $(b_{1},c_{1},d_{1},1), \\(b_{2},c_{2},d_{2},1)\in V_{1}$. Then there are two types of Berge $3$-paths: $(1,2,3,1)$-type and $(1,3,2,1)$-type. From the previous discussions, there are at most $216$ such Berge $3$-paths in $\mathcal{H}$.

If the given two vertices are in different parts, without loss of generality, suppose $(b_{1},c_{1},d_{1},1)\in V_{1}, (b_{2},c_{2},d_{2},2)\in V_{2}$. Then there are three types of Berge $3$-paths: $(1,2,1,2)$-type, $(1,2,3,2)$-type and $(1,3,1,2)$-type. From the previous discussions, there are at most $76$ such Berge $3$-paths in $\mathcal{H}$. Hence $\mathcal{H}$ is $\Theta_{3,217}^{B}$-free.

\section{Concluding remarks}\label{conclusion}
In this paper, we study the maximum number of edges in a $3$-uniform hypergraph with few Berge paths of length three between any two vertices. We determine the asymptotics for the Tur\'{a}n number of $\Theta_{3,217}^{B}$ via algebraic construction. Note that He and Tait \cite{TailBergetheta2018} showed that for fixed $\ell$ and $r$, there exists a large $t$ such that $\textup{ex}_{r}(n,\Theta_{\ell,t}^{B})=\Omega(n^{1+\frac{1}{\ell}}).$ However, the parameter $t$ might be possible to take $t=\ell^{O(\ell^{2})}$, and the random algebraic method falls well short of this. We believe that $217$ is not the best possible, hence improving the condition on $t$ will be interesting.

Our main technique to eliminate variables is using the resultant of polynomials.  To  show the power of this technique, we can give a new  proof for the following result, which has appeared in \cite{2019theta43}   by Verstra\"{e}te and Williford.
\begin{theorem}\cite{2019theta43}
$\textup{ex}(n,\theta_{4,3})=\Omega(n^{\frac{5}{4}})$.
\end{theorem}
\begin{proof}
Let $q$ be an odd prime power. The graph $G_{q}$ is defined on vertex set $V=\mathbb{F}_{q}^{4}$ such that  $u=(u_{1},u_{2},u_{3},u_{4})\in V$ is joined to $v=(v_{1},v_{2},v_{3},v_{4})\in V$ if $u\ne v$ and
\begin{align*}
&u_{2}+v_{2}=u_{1}v_{1},\\
&u_{3}+v_{4}=u_{1}v_{2}^{2},\\
&u_{4}+v_{3}=u_{1}^{2}v_{2}.
\end{align*}
It is easy to see that $G_{q}$ has $n:=q^{4}$ vertices and $\Omega(n^{\frac{5}{4}})$ edges.

Suppose that $G_{q}$ contains a $\theta_{4,3}$ with edges $\{au,uv,vw,wb,ax,xy,yz,zb,ad,de,ef,fb\}$. We first consider the octagon with edge set $\{au,uv,vw,wb,ax,xy,yz,zb\}$. By the definition of $G_{q}$, we have
\begin{align*}
&a_{2}+u_{2}=a_{1}u_{1}, && a_{3}+u_{4}=a_{1}u_{1}^{2}, && a_{4}+u_{3}=a_{1}^{2}u_{1},\\
&u_{2}+v_{2}=v_{1}u_{1}, && v_{3}+u_{4}=v_{1}u_{1}^{2}, && v_{4}+u_{3}=v_{1}^{2}u_{1},\\
&w_{2}+v_{2}=w_{1}v_{1}, && v_{3}+w_{4}=v_{1}w_{1}^{2}, && v_{4}+w_{3}=v_{1}^{2}w_{1},\\
&w_{2}+b_{2}=w_{1}b_{1}, && b_{3}+w_{4}=b_{1}w_{1}^{2}, && b_{4}+w_{3}=b_{1}^{2}w_{1},\\
&z_{2}+b_{2}=z_{1}b_{1}, && b_{3}+z_{4}=b_{1}z_{1}^{2}, && b_{4}+z_{3}=b_{1}^{2}z_{1},\\
&z_{2}+y_{2}=y_{1}z_{1}, && y_{3}+z_{4}=y_{1}z_{1}^{2}, && y_{4}+z_{3}=y_{1}^{2}z_{1},\\
&x_{2}+y_{2}=y_{1}x_{1}, && y_{3}+x_{4}=y_{1}x_{1}^{2}, && y_{4}+x_{3}=y_{1}^{2}x_{1},\\
&x_{2}+a_{2}=x_{1}a_{1}, && a_{3}+x_{4}=a_{1}x_{1}^{2}, && a_{4}+x_{3}=a_{1}^{2}x_{1}.
\end{align*}
From the left eight equations (the middle eight equations, the right eight equations, resp.), we can get the following equations.
\begin{align*}
f_{1}:=a_{1}u_{1}-u_{1}v_{1}+v_{1}w_{1}-w_{1}b_{1}+b_{1}z_{1}-z_{1}y_{1}+y_{1}x_{1}-x_{1}a_{1}=0,\\
f_{2}:=a_{1}u_{1}^{2}-u_{1}^{2}v_{1}+v_{1}w_{1}^{2}-w_{1}^{2}b_{1}+b_{1}z_{1}^{2}-z_{1}^{2}y_{1}+y_{1}x_{1}^{2}-x_{1}^{2}a_{1}=0,\\
f_{3}:=a_{1}^{2}u_{1}-u_{1}v_{1}^{2}+v_{1}^{2}w_{1}-w_{1}b_{1}^{2}+b_{1}^{2}z_{1}-z_{1}y_{1}^{2}+y_{1}^{2}x_{1}-x_{1}a_{1}^{2}=0.
\end{align*}

It is easy to see that if $r,s\in V$ are distinct and have a common neighbor, then $r_{1}\ne s_{1}$. Regarding $f_{i}$ as polynomials with variables $a_{1},u_{1},v_{1},w_{1},b_{1},z_{1},y_{1},x_{1}$, we can compute to get that
\begin{align*}
&R(f_{1},f_{2},u_{1})=g_{1}(a_{1}-v_{1}),\\
&R(f_{1},f_{3},u_{1})=g_{2}(a_{1}-v_{1}),\\
&R(g_{1},g_{2},x_{1})=(b_{1}-y_{1})(w_{1}-z_{1})^{2}(v_{1}-y_{1})(v_{1}-b_{1})(a_{1}-y_{1})(a_{1}-b_{1})(a_{1}-v_{1}+b_{1}-y_{1}).
\end{align*}
Hence, we have
\begin{align*}a_{1}+b_{1}=v_{1}+y_{1}.\end{align*}
 By the symmetry of octagon, we also have
 \begin{align}
 u_{1}+z_{1}=x_{1}+w_{1}.\label{eq100}
  \end{align}
  Since there are three octagons $\{au,uv,vw,wb,ax,xy,yz,zb\}$, $\{au,uv,vw,wb,ad,de,ef,fb\}$, $\{ax,xy,yz,zb,ad,de,ef,fb\}$ in $\theta_{4,3}$, we have $a_{1}+b_{1}=v_{1}+y_{1}=v_{1}+e_{1}=y_{1}+e_{1}$. Hence $v_{1}=y_{1}$ and $a_{1}+b_{1}=2v_{1}$. Substituting these equations into
\begin{align*}
f_{1}=a_{1}u_{1}-u_{1}v_{1}+v_{1}w_{1}-w_{1}b_{1}+b_{1}z_{1}-z_{1}y_{1}+y_{1}x_{1}-x_{1}a_{1}=0,
\end{align*}
we have $(a_{1}-v_{1})(u_{1}+w_{1}-x_{1}-z_{1})=0.$
Since $a,v$ have a common neighbor $u$, then $a_{1}\ne v_{1}$. Then we have $u_{1}+w_{1}=x_{1}+z_{1}$, from Eq.~(\ref{eq100}), we get that $u_{1}=x_{1}$, which contradicts to the fact that $u,x$ have a common neighbor $a$.
\end{proof}

It appears that our approach can be further applied to hypergraph Tur\'{a}n problems on $\Theta_{\ell,t}^{B}$ for other parameters. Designing more powerful algebraic constructions to deal with such problems is also of great interest.

\section*{Appendix}
\begin{program}
\begin{align*}
&P<x_1,x_2,x_3,y_1,y_3,z_2,z_3,t_3>:=PolynomialRing(RationalField(),8);\\
&f_1:=x_2*x_3-x_3*x_1+t_3*x_1-z_2*t_3+z_2*z_3-z_3*y_1+y_3*y_1-x_2*y_3;\\
&f_2:=x_2*x_3-x_2^2*x_3-z_2*t_3+z_2^2*t_3;\\
&f_3:=x_3*x_1-x_3^2*x_1-y_3*y_1+y_3^2*y_1;\\
&f_4:=z_3*y_1-z_3^2*y_1-t_3*x_1+t_3^2*x_1;\\
&f_5:=x_2*y_3-x_2^2*y_3-z_2*z_3+z_2^2*z_3;\\
&g_1:=Resultant(f_3,f_4,x_1);\\
&g_2:=Resultant(f_2,f_5,x_2);\\
&h:=Resultant(g_2,g_1,x_3);\\
&Factorization(h);
\end{align*}
\end{program}

\begin{thebibliography}{10}

\bibitem{Bollobas2008DM}
B.~Bollob\'{a}s and E.~Gy\H{o}ri.
\newblock Pentagons vs. triangles.
\newblock {\em Discrete Math.}, 308(19):4332--4336, 2008.

\bibitem{BondyEvenCycle1974}
J.~A. Bondy and M.~Simonovits.
\newblock Cycles of even length in graphs.
\newblock {\em J. Combin. Theory Ser. B}, 16:97--105, 1974.

\bibitem{BCP97}
W.~Bosma, J.~Cannon, and C.~Playoust.
\newblock The {M}agma algebra system. {I}. {T}he user language.
\newblock {\em J. Symbolic Comput.}, 24(3-4):235--265, 1997.

\bibitem{Brown1966}
W.~G. Brown.
\newblock On graphs that do not contain a thomsen graph.
\newblock {\em Canadian mathematical bulletin = Bulletin canadien de
  mathématiques}, 9(3), 1966.

\bibitem{ZilinCPC2017}
B.~Bukh and Z.~Jiang.
\newblock {A} bound on the number of edges in graphs without an even cycle.
\newblock {\em Combin. Probab. Comput.}, 26(1):1--15, 2017.

\bibitem{BukhTailtheta2018}
B.~Bukh and M.~Tait.
\newblock {T}ur\'{a}n number of theta graphs.
\newblock {\em arXiv preprint}, arXiv: 1804.10014, 2018.

\bibitem{Conlon2014}
D.~Conlon.
\newblock Graphs with few paths of prescribed length between any two vertices.
\newblock {\em arXiv preprint}, arXiv: 1411.0856, 2014.

\bibitem{1938erdos}
P.~Erd\H{o}s.
\newblock On sequences of integers no one of which divides the product of two
  others and on some related problems.
\newblock {\em Isvestia Nauchno-Issl. Inst. Mat. i Meh. Tomsk}, 2:74--82, 1938.

\bibitem{Erdos1966}
P.~Erd\H{o}s, A.~R\'{e}nyi, and V.~T. S\'{o}s.
\newblock On a problem of graph theory.
\newblock {\em Studia Sci. Math. Hungar.}, 1:215--235, 1966.

\bibitem{1946ErodsBAMS}
P.~Erd\H{o}s and A.~H. Stone.
\newblock On the structure of linear graphs.
\newblock {\em Bull. Amer. Math. Soc.}, 52:1087--1091, 1946.

\bibitem{Faudree1983}
R.~J. Faudree and M.~Simonovits.
\newblock On a class of degenerate extremal graph problems.
\newblock {\em Combinatorica}, 3(1):83--93, 1983.

\bibitem{Fuhrmann2012}
P.~A. Fuhrmann.
\newblock {\em A polynomial approach to linear algebra}.
\newblock Universitext. Springer, New York, second edition, 2012.

\bibitem{ImproveBergeK2t}
D.~Gerbner, A.~Methuku, and M.~Vizer.
\newblock Asymptotics for the {T}ur\'{a}n number of {B}erge-${K}_{2,t}$.
\newblock {\em J. Combin. Theory Ser. B}, 137:264--290, 2019.

\bibitem{GerbnerSIAM2017}
D.~Gerbner and C.~Palmer.
\newblock Extremal results for {B}erge hypergraphs.
\newblock {\em SIAM J. Discrete Math.}, 31(4):2314--2327, 2017.

\bibitem{Gyori2006CPC}
E.~Gy\H{o}ri.
\newblock Triangle-free hypergraphs.
\newblock {\em Combin. Probab. Comput.}, 15(1-2):185--191, 2006.

\bibitem{Lemons2012CPC}
E.~Gy\H{o}ri and N.~Lemons.
\newblock Hypergraphs with no cycle of a given length.
\newblock {\em Combin. Probab. Comput.}, 21(1-2):193--201, 2012.

\bibitem{TailBergetheta2018}
Z.~He and M.~Tait.
\newblock Hypergraphs with few berge paths of fixed length between vertices.
\newblock {\em arXiv preprint}, arXiv: 1807.10177, 2018.

\bibitem{JiangJCTB2018}
T.~Jiang and J.~Ma.
\newblock Cycles of given lengths in hypergraphs.
\newblock {\em J. Combin. Theory Ser. B}, 133:54--77, 2018.

\bibitem{LangWeil1954}
S.~Lang and A.~Weil.
\newblock Number of points of varieties in finite fields.
\newblock {\em Amer. J. Math.}, 76:819--827, 1954.

\bibitem{1995LowerBoundC2l}
F.~Lazebnik, V.~A. Ustimenko, and A.~J. Woldar.
\newblock A new series of dense graphs of high girth.
\newblock {\em Bull. Amer. Math. Soc. (N.S.)}, 32(1):73--79, 1995.

\bibitem{LV2003}
F.~Lazebnik and J.~Verstra\"{e}te.
\newblock On hypergraphs of girth five.
\newblock {\em Electron. J. Combin.}, 10:R25, 2003.

\bibitem{1907Mantel}
W.~Mantel.
\newblock Problem 28.
\newblock {\em Wiskundige Opgaven}, 10:60--61, 1907.

\bibitem{2012EX14}
T.~A. Terlep and J.~Williford.
\newblock Graphs from generalized {K}ac-{M}oody algebras.
\newblock {\em SIAM J. Discrete Math.}, 26(3):1112--1120, 2012.

\bibitem{Timmons2017}
C.~Timmons.
\newblock On {$r$}-uniform linear hypergraphs with no {B}erge-{$K_{2,t}$}.
\newblock {\em Electron. J. Combin.}, 24(4):Paper 4.34, 15, 2017.

\bibitem{1941Turan}
P.~Tur\'{a}n.
\newblock Eine extremalaufgabe aus der graphentheorie.
\newblock {\em Fiz Lapok}, pages 436--452, 1941.

\bibitem{Jacques2016}
J.~Verstra\"{e}te.
\newblock Extremal problems for cycles in graphs.
\newblock In {\em Recent trends in combinatorics}, volume 159 of {\em IMA Vol.
  Math. Appl.}, pages 83--116. Springer, [Cham], 2016.

\bibitem{2019theta43}
J.~Verstra\"{e}te and J.~Williford.
\newblock Graphs without theta subgraphs.
\newblock {\em J. Combin. Theory Ser. B}, 134:76--87, 2019.

\bibitem{WengerC4C6C10}
R.~Wenger.
\newblock Extremal graphs with no {$C^4$}'s, {$C^6$}'s, or {$C^{10}$}'s.
\newblock {\em J. Combin. Theory Ser. B}, 52(1):113--116, 1991.

\end{thebibliography}
\end{document}